\documentclass[oneside,11pt,reqno]{amsart}
\usepackage{amssymb,amsmath,amsthm,bbm,enumerate,mdwlist,url,multirow,hyperref,amsthm}
\usepackage[pdftex]{graphicx}
\usepackage[shortlabels]{enumitem}

\theoremstyle{definition}
\newtheorem{definition}{Definition}
\theoremstyle{theorem}

\newtheorem{theorem}[definition]{Theorem}

\numberwithin{equation}{section}
\theoremstyle{remark}
\newtheorem{remark}[definition]{Remark}
\newtheorem{example}[definition]{Example}
\def\PP{\mathsf P}
\def\pr{\mathrm{pr}}
\def\QQ{\mathsf Q}
\def\EE{\mathsf E}
\def\GG{\mathcal G}
\def\AA{\mathcal A}
\def\HH{\mathcal H}
\def\FF{\mathcal F}
\def\BB{\mathcal B}
\def\sgn{\mathrm{sgn}}
\begin{document}
\title{Observing a L\'evy process up to a stopping time}

\author{Matija Vidmar}
\address{Department of Mathematics, University of Ljubljana, and Institute of Mathematics, Physics and Mechanics, Slovenia}
\email{matija.vidmar@fmf.uni-lj.si}
\begin{abstract}
It is proved that the law of a possibly killed L\'evy process $X$, seen up to and including (resp.  up to strictly before) a stopping time, determines already the law of $X$ (resp. up to a compound Poisson component and killing). 
\end{abstract}

\thanks{Financial support from the Slovenian Research Agency is acknowledged (research core funding No. P1-0222). This research was conducted while the author was on a sabbatical at the University of Bath; he is grateful for its kind hospitality. The author also thanks Jean Bertoin and Jon Warren for  useful discussions on the topic of this paper.}



\maketitle

\section{Introduction}

We fix a $d\in \mathbb{N}$ --- the dimension of the Euclidean space $\mathbb{R}^d$ in which our L\'evy processes will live --- and a $\partial\notin \mathbb{R}^d$ -- it will play the role of a cemetery state. We agree that for $x\in \mathbb{R}^d\cup \{\partial\}$, $\partial\pm x$ and $x\pm \partial$ are all equal to $\partial$. 

Recall then that a stochastic process $X=(X_t)_{t\in [0,\infty)}$ on a probability space $(\Omega,\GG,\PP)$ is a possibly killed $\mathbb{R}^d$-valued L\'evy process in the filtration $\FF=(\FF_t)_{t\in [0,\infty)}$ satisfying $\FF_\infty\subset \GG$, if the following holds: 
(1) $X$ takes values in $(\mathbb{R}^d\cup \{\partial\},\mathcal{B}_{\mathbb{R}^d}\lor  \sigma(\{\partial\}))$ and is $\FF$-adapted;
(2) $X=\partial$ on $[\zeta,\infty)$, where  $\zeta:=\inf\{t\in (0,\infty):X_t=\partial\}$ is the lifetime of $X$;
(3) $X$ has paths that are right-continuous and have left limits on $[0,\zeta)$; 
(4) $\PP$-a.s. $X_0=0$ and $\zeta>0$; 
and
(5) $\PP[g(X_t-X_s)\mathbbm{1}_{\{s<\zeta\}}\vert \FF_s]=\mathbbm{1}_{\{s<\zeta\}}\PP[g(X_{t-s})]$ a.s.-$\PP$ for all real $0\leq s\leq t$ and all $g\in \mathcal{B}_{\mathbb{R}^d}/\mathcal{B}_{[0,\infty]}$ extended by $0$ on $\{\partial\}$.\footnote{Throughout we will write $\QQ[W]$ for $\EE_\QQ[W]$, $\QQ[W;A]$ for $\EE_\QQ[W\mathbbm{1}_A]$, $\QQ[W\vert \mathcal{H}]$ for $\EE_\QQ[W\vert \mathcal{H}]$, and $Z_\star \QQ$ for the law of $Z$ under $\QQ$ w.r.t. a $\sigma$-field on the codomain that will be clear from context or made explicit. For $\sigma$-fields $\AA$ and $\BB$, $\AA/\BB$ will denote the set of $\AA/\BB$-measurable maps; $\mathcal{B}_A$ is the Borel (under the standard topology) $\sigma$-field on $A$.}
When these conditions prevail, then in fact for any $\FF$-stopping time $T$ with $\PP(T<\zeta)>0$, under the measure $\PP(\cdot \vert T<\zeta)$, the process $\Delta_TX:=(X_{T+t}-X_T)_{t\in [0,\infty)}$ is independent of $\FF_T\vert_{\{T<\zeta\}}$ and has the same distribution as does $X$ under $\PP$. This is known as the strong Markov property of $X$. In particular there exists a necessarily unique $q\in [0,\infty)$ with $\zeta\sim_\PP\mathrm{Exp}(q)$. 

Furthermore, $X$ is called simply an $\mathbb{R}^d$-valued L\'evy process in $\FF$ if  $\zeta=\infty$ (corresponding to $q=0$). In the latter case, if $\mathsf{e}\in \mathcal{G}/\mathcal{B}_{[0,\infty]}$ is independent of $\FF_\infty$ and exponentially distributed (with strictly positive mean) under $\PP$, then $k_\mathsf{e}(X)$, the process $X$ killed at the time $\mathsf{e}$ (i.e. the process equal to $X$ on   $[0,\mathsf{e})$ and equal to $\partial$ on  $[\mathsf{e},\infty)$) is, in turn, a possibly killed L\'evy process in the progressive enlargement of $\FF$ by $\mathsf{e}$, i.e. in the smallest enlargement of $\FF$ that makes $\mathsf{e}$ a stopping time. Conversely, if we revert to $X$ being just a possibly killed L\'evy process, then there exists a unique law $\mathcal{L}$ of an $\mathbb{R}^d$-valued L\'evy process such that for all $t\in [0,\infty)$, $(X\vert_{[0,t]})_\star\PP(\cdot \vert t<\zeta)=(\xi\vert_{[0,t]})_\star \mathcal{L}$, where $\xi$ is the canonical process.

We refer the reader to \cite{sato} for further general theory and terminology concerning L\'evy processes (albeit without killing and in their natural filtrations). In particular, the reader will recall that, thanks to the stationary independent increments property, the one-dimensional distributions of a possibly killed L\'evy process determine already its law. 

Put differently, observing the laws of two possibly killed L\'evy processes $X^1$ and $X^2$ up to a (and even just at a given) strictly positive deterministic time, we are able to say whether or not $X^1$ and $X^2$ have the same law. The result of Theorem~\ref{theorem} below --- whose content was already described in informal terms in the abstract --- provides a non-obvious (cf. Examples~\ref{remark:essential-1} and~\ref{remark:essential-2}), though intuitively appealing complement to this observation, namely one in which a stopping time takes the role of a deterministic time. Remark~\ref{remark:markov} on p.~\pageref{remark:markov} will comment on the related case of continuous-time Markov chains. Finally, another motivation for the investigations --- and at the same time an application  --- of Theorem~\ref{theorem} is provided in Example~\ref{example} on p.~\pageref{example}. 

\section{Results and proofs}
Notation-wise, in the statement of the theorem to follow, for a process $Z=(Z_t)_{t\in T}$ on a probability space $(\Theta,\HH,\QQ)$, taking its values in $\mathbb{R}^d\cup \{\partial\}$, and defined temporally possibly only on some random subset $T$ of the time axis $[0,\infty)$, by  $Z_\star \QQ$ we mean the $\QQ$-law of the process $Z'$ that is equal to $Z$ on  $T$ and equal to some adjoined extra state $\uparrow$ on  $[0,\infty)\backslash T$, and we mean it on the space $((\mathbb{R}^d\cup \{\partial,\uparrow\})^{[0,\infty)},(\mathcal{B}_{\mathbb{R}^d}\lor \sigma(\{\partial\},\{\uparrow\}))^{\otimes [0,\infty)})$ [assuming of course $Z'$ is $\HH$-measurable w.r.t. the latter measurable structure]. Further, for laws $\mathcal{M}_1$, $\mathcal{M}_2$, $\mathcal{L}$ of $\mathbb{R}^d$-valued L\'evy processes, and for $q\in [0,\infty)$: (I) $\star$ denotes convolution of laws, viz. if $Y^1\sim_\QQ \mathcal{M}_1$ and $Y^2\sim_\QQ \mathcal{M}_2$, with $Y_1$ independent of $Y^2$ under $\QQ$, then $\mathcal{M}_1\star\mathcal{M}_2=(Y^1+Y^2)_\star\QQ$, and (II) $k_{q}$ is the operator of adding a killing at  rate $q$, viz. if $Y\sim_\QQ \mathcal{L}$ and $\mathsf{e}\sim_\QQ \mathrm{Exp}(q)$, $Y$ independent of $\mathsf{e}$ under $\QQ$, then $k_{q}(\mathcal{L})=(k_\mathsf{e}(Y))_\star \QQ$. 

Here is now the result of this note: 

\begin{theorem}\label{theorem}
For $i\in \{1,2\}$ let $X^i=(X_t^i)_{t\in [0,\infty)}$ be a possibly killed $\mathbb{R}^d$-valued L\'evy process, defined on a probability space $(\Omega^i,\GG^i,\PP^i)$ in the filtration $\mathcal{F}^i=(\FF_t^i)_{t\in [0,\infty)}$, and let $T^i$ be an $\FF^i$-stopping time with $\PP^i(T^i>0)>0$.
\begin{enumerate}[(i)]
\item \label{thm:i} If $(X^1\vert_{[0,T^1]\cap [0,\infty)})_\star\PP^1=(X^2\vert_{[0,T^2]\cap [0,\infty)})_\star \PP^2$, then ${X^1}_\star\PP^1={X^2}_\star\PP^2$.
\item\label{thm:ii} If $(X^1\vert_{[0,T^1)})_\star\PP^1=(X^2\vert_{[0,T^2)})_\star \PP^2$, then there exist a law $\mathcal{L}$ of a  L\'evy process, laws $\mathcal{L}_1$ and $\mathcal{L}_2$ of compound Poisson processes (allowing the zero process), and $\{q_1,q_2\}\subset [0,\infty)$, such that $(X^i)_\star\PP^i=k_{q_i}(\mathcal{L}\star \mathcal{L}_i)$ for $i\in \{1,2\}$, i.e. ``the laws of $X^1$ and $X^2$ differ only modulo compound Poisson processes and killing''. 
\end{enumerate}
\end{theorem}
Before giving the proof of this theorem,  some (counter)examples and comments. 
\begin{example}
Even if, for $i\in \{1,2\}$, $T^i$ is finite $\PP^i$-a.s., there can be no hope of having just $(X^1_{T^1}\mathbbm{1}_{\{T^1<\infty\}},T^1)_\star \PP^1=(X^2_{T^2}\mathbbm{1}_{\{T^2<\infty\}},T^2)_\star \PP^2$ imply ${X^1}_\star\PP^1={X^2}_\star\PP^2$. Indeed, if, on a common probability space, $B$ is a linear Brownian motion, $\mathbf{0}$ is the the zero process, and $S$ is the first hitting time of $0$ by $B$ after time $1$, then a.s. $S<\infty$, $S$ is a stopping time of the completed natural filtration of $B$ in which both $B$ and $\mathbf{0}$ are L\'evy processes,  $B_S\mathbbm{1}_{\{S<\infty\}}=0=\mathbf{0}_S\mathbbm{1}_{\{S<\infty\}}$ a.s., yet of course $B$ and $\mathbf{0}$ do not have the same law. 
\end{example}
\begin{example}\label{remark:essential-1}
For \ref{thm:i} the stopping time property is essential. If $\mathbf{0}$ is the zero process and $N$ is a homogeneous Poisson process, both defined on a common probability space, then letting $S$ be the first jump time of $N$, one has $\mathbf{0}=0=N$ on $[0,S/2]\cap [0,\infty)$ and $S/2>0$ a.s., yet $\mathbf{0}$ and $N$ do not have the same law. 
\end{example}
\begin{example}\label{remark:essential-2}
Also for \ref{thm:ii} the stopping time property is essential. Indeed, by a result of Williams \cite[Theorem~55.9]{rogers}, for any given $c\in (0,\infty)$, on a common probability space, one may construct a Brownian motion with drift $-c$, $B^1$, a Brownian motion with drift $c$, $B^2$, and an exponentially distributed random time $\gamma$ of rate $2c$, such that $\gamma$ is independent of $B^1$, is equal to the time of the overall infimum of $B^2$, and with $B^1\vert_{[0,\gamma)}=B^2\vert_{[0,\gamma)}$. (Of course this is also another counterexample for \ref{thm:i}.)
\end{example}
\begin{example}
In \ref{thm:i}, even if $(\Omega^1,\GG^1,\PP^1)=(\Omega^2,\GG^2,\PP^2)$, $\FF^1=\FF^2$, and a.s.  $T^1=T^2$ \& $X^1=X^2$ on $[0,T^1]\cap [0,\infty)$, still the conclusion cannot be strengthened to a.s. equality. To exemplify this, take, on a common probability space, a standard one-dimensional  Brownian motion $B^1$, a random time $T$ independent of $B^1$, positive and finite with a positive probability, and let $B^2$ be got from $B^1$ by changing $B^1$ into an independent standard linear Brownian motion after time $T$. Then $B^1$ and $B^2$ are both standard univariate Brownian motions in their completed joint natural filtration of which $T$ is  a stopping time, they a.s. agree on $[0,T]\cap [0,\infty)$, but they are not a.s. equal.
\end{example}
\begin{example}
The conclusion of \ref{thm:ii} cannot be improved. For instance, if, on  a common probability space, $N$ is a homogeneous Poisson process, while $M$ is zero up to and then killed at the first jump time $S$ of $N$, then $N$ and $M$ are L\'evy processes in the completed natural filtration of $N$ of which $S$ is a stopping time, a.s. $N=0=M$ on $[0,S)$ and $S>0$, yet $N$ and $M$ ``differ by killing and by a compound Poisson process''. 
\end{example}
\begin{remark}\label{remark:random-walk}
The content of Theorem~\ref{theorem} makes sense also for a possibly killed  random walk (in the obvious interpretation of that qualification), but in that case it is trivial. Indeed, if $Z=(Z_n)_{n\in \mathbb{N}_0}$ is a possibly killed $\mathbb{R}^d$-valued random walk, in a filtration $\HH=(\HH_n)_{n\in \mathbb{N}_0}$, under a probability $\QQ$, and if $S$ is a stopping time of $\HH$ that is positive with a positive $\QQ$-probability, then $Z$ is independent under $\QQ$ of $\HH_0\ni \{S>0\}$. Therefore, if the $\QQ$-law of $Z\vert_{[0,S]\cap \mathbb{N}_0}$ is known, then for any $f\in \mathcal{B}_{\mathbb{R}^d}/\mathcal{B}_{[0,\infty]}$ extended by $0$ on $\{\partial\}$, the quantity $\QQ[f(Z_1)]\QQ(S>0)=\QQ[f(Z_1);S>0]=\QQ[f(Z_1^S);S>0]$, and hence the $\QQ$-law of $Z$ is known. On the other hand the knowledge of the $\QQ$-law of $Z\vert_{[0,S)\cap \mathbb{N}_0}$ clearly need not determine the $\QQ$-law of $Z$ at all, since one can take $S=1$.
\end{remark}
\begin{remark}
\ref{thm:i} implies that, for any killed (lifetime $\zeta<\infty$ a.s.) L\'evy process $X$ in a filtration $\FF$ under a measure $\PP$, and any $\FF$-stopping time $S$ with $\PP(S>0)>0$, one has $\PP(S\geq \zeta)>0$.  
\end{remark}
We turn now to the
\begin{proof}[Proof of Theorem~\ref{theorem}.]
Let $i\in \{1,2\}$. Replacing  both $T^i$ with $T^i\land 1$ if necessary we may assume each $T^i$ is finite.  Then take the product space $(\Omega^{i,\infty},\GG^{i,\infty},\PP^{i,\infty}):=((\Omega^i)^\mathbb{N},(\GG^i)^{\otimes \mathbb{N}},\times_\mathbb{N}\PP^i))$. Set $X^{i,n}:=X^i\circ \pr_n$  and $T_n^i:=T^i\circ \pr_n$ for $n\in \mathbb{N}$. By the law of large numbers, since $\PP^i(T^i>0)>0$, and discarding a negligible set if necessary, we may assume that $S_n^i:=\sum_{k=1}^nT_k^i\uparrow \infty$ as $n\to\infty$ over $\mathbb{N}_0$. 

Next we define the process $Y^{i,n}=(Y^{i,n}_t)_{t\in [0,\infty)}$ on $\Omega^{i,\infty}$, with $n\in \mathbb{N}$, as follows: $Y^{i,1}:=X^{i,1}$ and then inductively, for $n\in \mathbb{N}$, $Y^{i,n+1}=Y^{i,n}$ on $[0,S_n^i)$, while $Y^{i,n+1}=Y^{i,n}_{S_n^i}+X^{i,n+1}_{\cdot-S_n^i}$ on $[S_n^i,\infty)$. In words, still for $n\in \mathbb{N}$, $Y^{i,n}$ starts at $X^{i,1}_0$ and then, up to hitting $\partial$, for $k\in \{1,\ldots,n-1\}$, the increments of $Y^{i,n}$ on $[S_{k-1}^i,S_k^i]$ are those of $X^{i,k}$ on $[0,T_k^i]$, while on $[S_n^i,\infty)$, they are those of $X^{i,n}$ on $[0,\infty)$. 

It is then clear that, as $n\to\infty$, the $Y^{i,n}$ are converging pointwise  to a process, that we denote by $Y^i$; we claim furthermore, that for each $n\in \mathbb{N}$, $Y^{i,n}$ (and therefore, in the limit, $Y^i$) has the same law under $\PP^{i,\infty}$ as does $X^i$ under $\PP^i$. 


We need only prove the latter for ``$i=1$'' (it is the same for ``$i=2$''); and then we drop, in the next paragraph only, the superscript ``$i=1$'' to ease the notation.

Take then $\{n,k\}\subset \mathbb{N}$, $\{g_1,\ldots,g_k\}\subset \mathcal{B}_{\mathbb{R}^d}/\mathcal{B}_{[0,\infty]}$ extended by zero on $\{\partial\}$, and real numbers $0=t_0<\cdots <t_k$; we are to show that 
 $\PP[g_1(X_{t_1})\cdots g_k(X_{t_k})]=\PP^\infty[g_1(Y^{n}_{t_1})\cdots g_k(Y^{n}_{t_k})]$. We compute:  $$\PP^\infty[g_1(Y^{n+1}_{t_1})\cdots g_k(Y^{n+1}_{t_k})]$$ \footnotesize
$$=\PP^\infty[g_1(Y^{n}_{t_1})\cdots g_k(Y^{n}_{t_k});t_k\leq S_n]+\sum_{l=1}^k\PP^\infty[g_1(Y^{n}_{t_1})\cdots g_{l-1}(Y^n_{t_{l-1}}) g_l(Y^{n+1}_{t_l})\cdots g_k(Y^{n+1}_{t_k});t_{l-1}\leq S_n< t_l]$$\normalsize
$$=\PP^\infty[g_1(Y^{n}_{t_1})\cdots g_k(Y^{n}_{t_k});t_k\leq S_n]+$$
$$\sum_{l=1}^k\PP^\infty[g_1(Y^{n}_{t_1})\cdots g_{l-1}(Y^n_{t_{l-1}}) g_l(Y^{n}_{S_n}+X^{n+1}_{t_l-S_n})\cdots g_k(Y^{n}_{S_n}+X^{n+1}_{t_k-S_n});t_{l-1}\leq S_n< t_l,T_n<\zeta_n]$$\normalsize
$$=\PP^\infty[g_1(Y^{n}_{t_1})\cdots g_k(Y^{n}_{t_k});t_k\leq S_n]+$$\footnotesize
$$\sum_{l=1}^k\PP^\infty[g_1(Y^{n}_{t_1})\cdots g_{l-1}(Y^n_{t_{l-1}}) g_l(Y^{n}_{S_n}+(\Delta_{T_n}X^n)_{t_l-S_n})\cdots g_k(Y^{n}_{S_n}+(\Delta_{T_n}X^n)_{t_k-S_n});t_{l-1}\leq S_n< t_l,T_n<\zeta_n]$$\normalsize
$$=\PP^\infty[g_1(Y^{n}_{t_1})\cdots g_k(Y^{n}_{t_k});t_k\leq S_n]+\sum_{l=1}^k\PP^\infty[g_1(Y^{n}_{t_1})\cdots g_{l-1}(Y^n_{t_{l-1}}) g_l(Y^{n}_{t_l})\cdots g_k(Y^{n}_{t_k});t_{l-1}\leq S_n< t_l]$$
$$=\PP^\infty[g_1(Y^{n}_{t_1})\cdots g_k(Y^{n}_{t_k})],$$
where crucially in the third equality: we  used (I) the strong Markov property for $X$ at time $T$, plus the various independences coming from the construction of $\PP^\infty$, to establish that under $\PP^\infty$, conditionally on $\{T_n<\zeta_n\}$, the process $\Delta_{T_n}X^n$ has the same law as $X^{n+1}$ and is, like $X^{n+1}$, independent of $(X^1,\ldots,X^{n-1},(X^n)^{T_n},T_1,\ldots,T_n)$; noting that (II) $Y^n_{S_n}=X^1_{T_1}+\cdots+X^n_{T_n}$, $S_n=T_1+\cdots+T_n$, and for each $l\in \{1,\ldots,k\}$ and $m\in \{1,\ldots,l-1\}$, $g_m(Y^n_{t_m})\mathbbm{1}_{\{t_{l-1}\leq S_n\}}$ are all measurable w.r.t. $\sigma(X^1,\ldots,X^{n-1},(X^n)^{T_n},T_1,\ldots,T_n)$: it is only not obvious for the latter -- to check it, write $Y^n_{t_m}\mathbbm{1}_{\{t_{l-1}\leq S_n\}}=\sum_{w=1}^nY^n_{t_m}\mathbbm{1}_{\{S_{w-1}<t_m\leq S_w,t_{l-1}\leq S_n\}}=\sum_{w=1}^n(X^1_{T_1}+\cdots +X^{w-1}_{T_{w-1}}+(X^w)^{T_w}_{t_m-S_{w-1}})\mathbbm{1}_{\{S_{w-1}<t_m\leq S_w,t_{l-1}\leq S_n\}}$.\footnote{It is tempting to think that one could somehow bypass the strong Markov property and still prove that $Y_\star\PP^\infty=X_\star\PP$ without assuming that $T$ is an $\FF$-stopping time. But it is false. For instance, if, under $\PP$, $X$ is a linear Brownian motion with strictly negative drift, and if $T$ is the last time that $X$ is at $0$, then under $\PP^\infty$, $0$ is recurrent for the process $Y$, while it is transient for the process $X$ under $\PP$.} An inductive argument allows to conclude.


As we have noted, this now establishes that, for $i\in \{1,2\}$, ${Y^i}_\star\PP^{i,\infty }={X^i}_\star\PP^i$; it will also be helpful to keep in mind that, up to hitting $\partial$, for $k\in \mathbb{N}$, the increments of $Y^{i}$ on $[S_{k-1}^i,S_k^i]$ are those of $X^{i,k}$ on $[0,T_k^i]$.

\ref{thm:i}. $Y^1$ and $Y^2$ are seen to be the same measurable transformation of the sequences  $((X^{i,n}_t)_{t\in [0,T_n^i]})_{n\in \mathbb{N}}$ with $i=1$ and $i=2$, respectively. Furthermore, for $i\in \{1,2\}$, under $\PP^{i,\infty}$, by construction, the sequence $((X^{i,n}_t)_{t\in [0,T_n^i]})_{n\in \mathbb{N}}$ consists of i.i.d. random elements. Besides, $((X^{1,1}_t)_{t\in [0,T_1^1]})_\star \PP^{1,\infty}=((X^{1}_t)_{t\in [0,T^1]})_\star \PP^{1}=((X^{2}_t)_{t\in [0,T^2]})_\star \PP^{2}=((X^{2,1}_t)_{t\in [0,T_1^2]})_\star \PP^{2,\infty}$ by the assumption of \ref{thm:i}. Therefore ${X^1}_\star\PP^1={Y^1}_\star\PP^{1,\infty }={Y^2}_\star\PP^{2,\infty}={X^2}_\star\PP^2$.

\ref{thm:ii}. Let $i\in \{1,2\}$ and denote by $\zeta^i$ the lifetime of $X^i$. Replacing $T^i$ with $T^i\land \zeta^i=\inf\{t\in (0,T^i):X^i_t=\partial\}\land T^i$ if necessary we may assume that $T^i\leq \zeta^i$; then (possibly by enlarging the underlying space and filtration) we may assume that $\zeta^i=\infty$. 
Let next $\nu^i$ be the L\'evy measure of $X^i$ under $\PP^i$ (and hence of $Y^i$ under $\PP^{i,\infty}$). Set $\gamma:=\nu^1+\nu^2$, $f^i:=\frac{d\nu^i}{d\gamma}$ for $i\in \{1,2\}$, and $\nu:=(f^1\land f^2)\cdot \gamma$. We check that (A) $\nu^i-\nu$, $i\in\{1,2\}$, are finite measures. 

Suppose per absurdum, and then without loss of generality, that $\nu^{1\prime}:=\nu^1-\nu=(f^1-f^1\land f^2)\cdot \gamma$ is infinite. The measure $\nu^{1\prime}$ is locally finite in $\mathbb{R}^d\backslash \{0\}$, hence there exists a sequence $(A_k)_{k\in \mathbb{N}}$ in $\mathcal{B}_{\mathbb{R}^d\backslash \{0\}}$ such that $A_k\subset \{f^1>f^1\land f^2\}=\{f^1>f^2\}$ for each $k\in \mathbb{N}$ and such that $[0,\infty)\ni \nu^{1\prime}(A_k)=\nu^1(A_k)-\nu^2(A_k)\uparrow \infty$ as $k\to \infty$. For $k\in \mathbb{N}$ and $i\in \{1,2\}$, set $\xi_k^i$ equal to the number of jumps of $Y^i$ during the time interval $(0,1]$ that fall into the Borel set $A_k$; then  $\xi_k^i\sim_{\PP^{i,\infty}} \mathrm{Pois}_{\mathbb{N}_0}(\nu^i(A_k))$ and in particular $\PP^{i,\infty}[\xi_k^i]=\nu^i(A_k)$. Consequently, for each $k\in \mathbb{N}$, $\PP^{1,\infty}[\xi_k^1]-\PP^{2,\infty}[\xi_k^2]=\nu^1(A_k)-\nu^2(A_k)$, which is $\uparrow \infty$ as $k\to \infty$. On the other hand, setting $N:=\sum_{k\in \mathbb{N}}\mathbbm{1}_{\{S_k^1\leq  1\}}$, it is clear by construction of the processes $Y^1$ and $Y^2$ and from $(X^1\vert_{[0,T^1)})_\star\PP^1=(X^2\vert_{[0,T^2)})_\star \PP^2$, that $\PP^{1,\infty}[\xi_k^1]-\PP^{2,\infty}[\xi_k^2]\leq \PP^{1,\infty}[N]$ for all $k\in\mathbb{N}$. At the same time, the $T_k^1$, $k\in \mathbb{N}$, are i.i.d. under $\PP^{1,\infty}$, hence by renewal theory \cite{renewal}, since $\PP^{1,\infty}(T_1^1>0)=\PP^1(T^1>0)>0$, it follows that $\PP^{1,\infty}[N]<\infty$, a contradiction.

Next, by the L\'evy-It\^o decomposition one can write, for each $i\in \{1,2\}$, $\PP^i$-a.s.: $$X^i_t=B^i_t+\Gamma^it+\lim_{\epsilon \downarrow 0}\int_{(0,t]\times B(\epsilon,1]}x[J^i(ds,dx)-ds\nu^i(dx)]+\int_{(0,t]\times B(1,\infty)}xJ^i(ds,dx),\quad t\in [0,\infty),$$
for a $d$-dimensional (possibly non-standard, of course) $\PP^i$-Brownian motion $B^i$, a $\Gamma^i\in \mathbb{R}^d$, and with $J^i$ being the Poisson random measure of the jumps of $X^i$. Furthermore, by (A), the limit $L:=\lim_{\epsilon\downarrow 0}\left[\int_{B(\epsilon,1]}x\nu^1(dx)-\int_{B(\epsilon,1]}x\nu^2(dx)\right]$ is well-defined in $\mathbb{R}^d$. Set $X^{1,c}_t:=B^1_t+\Gamma^1t$  and $X^{2,c}_t:=B^2_t+\Gamma^2t+Lt$ for $t\in [0,\infty)$. Then, with $i\in \{1,2\}$, on the time interval $[0,T^i)$, the processes $X^{i,c}$ can be extracted from the processes $X^i$ by the same measurable transformation (because this is true of the jumps). Therefore $(X^{1,c}\vert_{[0,T^1)})_\star\PP^1=(X^{2,c}\vert_{[0,T^2)})_\star\PP^2$ and hence by sample-path continuity $(X^{1,c}\vert_{[0,T^1]\cap [0,\infty)})_\star\PP^1=(X^{2,c}\vert_{[0,T^2]\cap [0,\infty)})_\star\PP^2$. Now, $X^{1,c}$ and $X^{2,c}$ are still L\'evy processes in the filtrations $\FF^1$ and $\FF^2$, respectively. Thus, by part \ref{thm:i}, (B) ${X^{1,c}}_\star \PP^1 ={X^{2,c}}_\star\PP^2$.

Combining (A)-(B), by the independence between the ``jump'' and ``continuous'' part present in the L\'evy-It\^o decomposition, the desired conclusion follows. 
\end{proof}

\begin{remark}\label{remark:markov}
The proof of Theorem~\ref{theorem}\ref{thm:i} can be tweaked to handle the case of continuous-time Markov chains, though the result is less definitive in this context.  Let us look at this in more detail. 

Fix a countable set $E$  -- it will be the state space; fix also --- it will be the cemetery state --- a $\partial\notin E$. Recall then that a process $X=(X_t)_{t\in [0,\infty)}$, defined on a measurable space $(\Omega,\mathcal{G})$, is a (minimal) continuous-time $E$-valued Markov chain, in a filtration $\FF=(\FF_t)_{t\in [0,\infty)}$ with $\FF_\infty\subset \GG$, under the probabilities $\PP=(\PP_x)_{x\in E}$ on $(\Omega,\mathcal{G})$, provided: (1) $X$ takes  values in $E\cup \{\partial\}$, endowed with the discrete topology and measurable structure, and it is $\FF$-adapted; (2) $X$ has paths that are right-continuous, $X$ is $E$-valued on $[0,\zeta)$, and $X=\partial$ on $[\zeta,\infty)$, where $\zeta:=\lim_{n\to\infty}J_n$, $(J_n)_{n\in \mathbb{N}}$ being the sequence of the consecutive jump times of $X$; (3) $\PP_x(X_0=x)=1$ for all $x\in E$; (4) for $\{t,s\}\subset [0,\infty)$, $x\in E$, and $g\in 2^E/\mathcal{B}_{[0,\infty]}$ extended by $0$ on $\{\partial\}$, one has $\PP_x[g(X_{t+s})\vert \FF_t]=\PP^{X_t}[g(X_s)]$ a.s.-$\PP_x$ on $\{t<\zeta\}$. As is well-known, such an $X$ then has the strong Markov property: for any $\FF$-stopping time $T$, $x\in E$ and $g\in (2^{E\cup \{\partial\}})^{\otimes [0,\infty)}/\mathcal{B}_{[0,\infty]}$, $\PP_x[g(X_{T+\cdot})\vert \FF_T]=\PP^{X_T}[g(X)]$ a.s.-$\PP_x$ on $\{T<\zeta\}$.


Suppose then that the system $(\Omega,\GG,\FF,\PP,X)$ constitutes such a continuous-time Markov chain and let $T$ be an $\FF$-stopping time. Take the measure $\PP^\infty:=\times_{n\in \mathbb{N}}\times_{x\in E}\PP_x$ on $(\prod_{n\in \mathbb{N}}\prod_{x\in E}\Omega,(\mathcal{G}^{\otimes E})^{\otimes \mathbb{N}})$ and let $X^{x,n}:=X\circ\pr_x\circ \pr_n$, $T^{x,n}:=T\circ\pr_x\circ \pr_n$ for $x\in E$, $n\in \mathbb{N}$ (in words, we take denumerably many independent copies of $X$ for each starting position). Additionally set $X^{\partial,n}\equiv \partial$ and $T^\partial\equiv\infty$ for $n\in \mathbb{N}$. Fix next an $x\in E$. Define $Y^{x,1}:=X^{x,1}$, $S_1^x:=T^{x,1}$, and then recursively, for $n\in \mathbb{N}$, $Y^{x,n+1}:=Y^{x,n}$ on $[0,S_n^x)$, $Y^{x,n+1}:=X^{Y^{x,n}_{S_n^x},n+1}_{\cdot- S_n^x}$ on $[S_n^x,\infty)$, and $S_{n+1}^x:=S_n^x+T^{Y^{x,n}_{S_n^x},n+1}$. (Note the denumerable state space ensures suitable measurability of these objects and it ensures that $\PP^\infty$-a.s. $Y^{x,n+1}_{S_n^x}=Y^{x,n}_{S_n^x}$ for all $n\in \mathbb{N}$.) Set furthermore $S^x:=\lim_{n\to\infty}S_n^x$ and assume that $\PP^\infty(S^x=\infty)=1$. (It would be an interesting question in its own right to investigate under which conditions does $\PP^\infty(S^x=\infty)=1$ in fact obtain, however this will not be pursued here.)
%
%
It is then clear that the $Y^{x,n}$ converge $\PP^\infty$-a.s. to a process $Y^x$ as $n\to\infty$. Furthermore, using the strong Markov property, similarly to how we did in the proof of the L\'evy case, we may show that, for each $n\in \mathbb{N}$, $Y^{x,n}$, and hence in the limit $Y^x$, has the same law under $\PP^\infty$ as does $X$ under $\PP_x$. We leave the (grantedly more tedious when compared to the L\'evy case) details of these computations to the interested reader. 

As a consequence of the preceding we obtain then, just as in the proof of Theorem~\ref{theorem}\ref{thm:i}, the statement:

\begin{quote}
For $i\in \{1,2\}$, let $(\Omega^i,\GG^i,\FF^i,\PP^i,X^i)$ be a continuous-time $E$-valued Markov chain in the sense made precise above, and let $T^i$ be an $\FF^i$-stopping time; associate to it $\PP^{i,\infty}$ and the times $(S^{i,x})_{x\in E}$ in the obvious way, as above. Assume $(X^1\vert_{[0,T^1]\cap [0,\infty)})_\star\PP^{1}_x=(X^2\vert_{[0,T^2]\cap [0,\infty)})_\star\PP^{2}_x$ for all $x\in E$. Let $x\in E$. If $\PP^{1,\infty}(S^{1,x}=\infty)=1$  (and hence $\PP^{2,\infty}(S^{2,x}=\infty)=1$), then ${X^1}_\star \PP^1_x={X^2}_\star \PP^2_x$.
\end{quote}
The case of discrete-time Markov chains is again trivial in this context (cf. Remark~\ref{remark:random-walk}). The analogue of Theorem~\ref{theorem}\ref{thm:ii} is of no interest in the context of Markov chains.
\end{remark}

\begin{example}\label{example}
 We close this paper with an example in the context of self-similar Markov processes  in which Theorem~\ref{theorem} produces non-trivial information. 

To this end, let $X=(X_t)_{t\in [0,\infty)}$ be a one-dimensional stable L\'evy process under the probabilities $(\PP_x)_{x\in \mathbb{R}}$ in a filtration $\FF=(\FF_t)_{t\in [0,\infty)}$ satisfying the usual hypotheses. We make the standing assumption that $X$ has jumps of both signs in order to avoid triviality, and refer the reader to \cite[Chapter~13]{kyprianou} \cite{alea} \cite[Section~2]{deep-1} for any unexplained terminology and facts that we shall state without proof below: introducing everything properly here would not be consistent with the scope of this paper.

We consider the  following two processes: $Y$, which is $X$ sent to $0$ on hitting $(-\infty,0]$ (and then stopped); and  $Z$, which is $X$ sent to $0$ on hitting $0$ (and then stopped). It is  then well-known that $Y$ is a positive self-similar Markov process under the probabilities $(\PP_x)_{x\in (0,\infty)}$ and that $Z$ is a real self-similar Markov process under the probabilities $(\PP_x)_{x\in \mathbb{R}\backslash \{0\}}$, both in the filtration $\FF$. Moreover, defining $T_0:=\inf \{s\in (0,\infty):Y_s=0\}$, $R_0:=\inf \{s\in (0,\infty):Z_s=0\}$,  $$\tau_t:=\inf\left\{s\in (0,R_0):\int_0^s\vert Z_v\vert^{-\alpha}dv>t\right\}\land R_0$$ and $\GG_t:=\FF_{\tau_t}$ for $t\in [0,\infty)$, we have as follows:
\begin{enumerate}[(1)]
\item Put $\gamma:=\int_0^{R_0}\vert Z_v\vert^{-\alpha}dv$ and define the processes $\mu=(\mu_t)_{t\in [0,\infty)}$ and $J=(J_t)_{t\in [0,\infty)}$, by setting $\mu_t:=\log(\vert Z_{\tau_t}\vert)\text{ and }J_t:=\sgn(Z_{\tau_t})$ for $t\in (0,\gamma),$
$\mu$ and $J$ being killed on the time-interval $[\gamma,\infty)$. Then \cite{rivero} $(\mu,J)$ is a possibly killled Markov additive process (MAP) under the probabilities $(\PP_{je^x})_{(x,j)\in \mathbb{R}\times \{-1,1\}}$ in the filtration $(\GG_t)_{t\in [0,\infty)}$. This is known as the Lamperti-Kiu transform. Further, let $(G,M)$ be the ascending ladder MAP of $(\mu,J)$ with associated local time at the maximum $L$; then, under $\PP_1$, $G$ killed at the first time $M$ changes its sign is a killed subordinator in the filtration $(\GG_{L^{-1}_s})_{s\in [0,\infty)}$. It is the killed L\'evy process $H^1$ which describes the movement of $G$ as long as the modulating chain $M$ is in state $1$.
\item Similarly put $\zeta:=\int_0^{T_0}Y_v^{-\alpha}dv$  and define the process $\xi=(\xi_t)_{t\in [0,\infty)}$ by setting $\xi_t:=\log(Y_{\tau_t})$ for $t\in (0,\zeta),$
$\xi$ being killed on the time-interval $[\zeta,\infty)$. Then  \cite{lamperti} $\xi$ is a possibly killled L\'evy process in the filtration $(\GG_t)_{t\in [0,\infty)}$ under the probabilities $(\PP_{e^x})_{x\in\mathbb{R}}$. This is known as the Lamperti transform. Further, let $H$ be the ascending ladder height process of $\xi$ under a normalization of the local time at the maximum that is consistent with that of $L$; then, under $\PP_1$, $H$ is a killed subordinator in the filtration $(\GG_{L^{-1}_s})_{s\in [0,\infty)}$.
\end{enumerate}

Besides, it is clear from the pathwise construction of $Y$ and $Z$, that $H$ and $H^1$ agree on $[0,L_{\zeta})$ with $L_{\zeta}$ being a stopping time of $(\GG_{L^{-1}_s})_{s\in [0,\infty)}$. Indeed $L_{\zeta}$ is the lifetime of $H$, and albeit it is not the lifetime of $H^1$, it is certainly majorized by the latter. 

It follows then from Theorem~\ref{theorem} that the laws of $H$ and $H^1$ differ only by killing and compound Poisson components, which yields an a priori insight into the non-trivial probabilistic structure of the MAP $(G,M)$ in terms of the much simpler object $H$. (A fully explicit description of the  law of $(G,M)$, viz. of the MAP exponent of $(G,M)$, is non-trivial, see \cite{deep-1}.) 
\end{example}



\bibliographystyle{plain}
\bibliography{Biblio_filtrations}
\end{document}